\documentclass[12pt]{amsart}

\usepackage[english]{babel}

\usepackage[
pdftex,hyperfootnotes]{hyperref}
\hypersetup{
 colorlinks=true,
 linkcolor=NavyBlue, 
 urlcolor=RoyalPurple,
 citecolor=OliveGreen,
 pdftitle={Positive bidiagonal factorization of tetradiagonal Hessenberg matrices},
 bookmarks=true,
 pdfpagemode=FullScreen,
}

\usepackage[usenames,dvipsnames,svgnames,table,x11names]{xcolor}
\usepackage{pgfplots}
\usepackage{tikz}
\usepackage{tikz-3dplot}
\usetikzlibrary{automata,quotes, chains,matrix,calc,shadows,shapes.callouts,shapes.geometric,shapes.misc,positioning,patterns,decorations.shapes,
 decorations.pathmorphing,decorations.markings,decorations.fractals,decorations.pathreplacing,shadings,fadings,arrows.meta,bending}

\usepackage{nicematrix}
\usepackage[utf8]{inputenc}
\usepackage{polski}
\usepackage{comment}
\usepackage[textwidth=17.5cm,textheight=22.275cm,
height=
24.275cm,
width=18cm]{geometry}

\usepackage{rotating,multirow,pdflscape}
\usepackage{amssymb,latexsym,amsmath,amsthm,bm}
\usepackage{mathrsfs}
\usepackage{mathtools,arydshln,mathdots}
\usepackage{bigints}
\makeatletter
\renewcommand*\env@matrix[1][*\c@MaxMatrixCols c]{%
 \hskip -\arraycolsep
 \let\@ifnextchar\new@ifnextchar
 \array{#1}}
\makeatother

\mathtoolsset{centercolon}
\usepackage[table]{xcolor}
\usepackage[toc,page]{appendix}

\usepackage{tocvsec2}

\usepackage{Baskervaldx}
\usepackage[]{newtxmath }
\newtheorem{coro}{Corollary}

\newtheorem{defi}{Definition}
\newtheorem{teo}{Theorem}

\newtheorem{pro}{Proposition}
\newtheorem{lemma}{Lemma}

\newcommand{\diag}{\operatorname{diag}}

\newcommand{\N}{\mathbb{N}}
\newcommand{\R}{\mathbb{R}}

\allowdisplaybreaks[1]

\makeatletter
\DeclareRobustCommand{\gaussk}{\DOTSB\gaussk@\slimits@}
\newcommand{\gaussk@}{\mathop{\vphantom{\sum}\mathpalette\bigcal@{K}}}

\newcommand{\bigcal@}[2]{%
 \vcenter{\m@th
 \sbox\z@{$#1\sum$}%
 \dimen@=\dimexpr\ht\z@+\dp\z@
 \hbox{\resizebox{!}{0.8\dimen@}{$\mathcal{K}$}}%
 }%
}
\newcommand{\cfracplus}{\mathbin{\cfracplus@}}
\newcommand{\cfracplus@}{%
 \sbox\z@{$\dfrac{1}{1}$}%
 \sbox\tw@{$+$}%
 \raisebox{\dimexpr\dp\tw@-\dp\z@\relax}{$+$}%
}
\newcommand{\cfracdots}{\mathord{\cfracdots@}}
\newcommand{\cfracdots@}{%
 \sbox\z@{$\dfrac{1}{1}$}%
 \sbox\tw@{$+$}%
 \raisebox{\dimexpr\dp\tw@-\dp\z@\relax}{$\cdots$}%
}
\makeatother
\begin{document}

\title[Positive bidiagonal factorization of tetradiagonal matrices]
{Positive bidiagonal factorization \\of tetradiagonal Hessenberg matrices}

\author[A Branquinho]{Amílcar Branquinho$^{1,\flat}$}
\address{$^1$CMUC,
Departamento de Matemática, Universidade de Coimbra, 3001-454 Coimbra, Portugal}
\email{ajplb@mat.uc.pt}

\author[A Foulquié]{Ana Foulquié-Moreno$^{2,\natural}$}
\address{$^2$CIDMA,
Departamento de Matemática, Universidade de Aveiro, 3810-193 Aveiro, Portugal}
\email{foulquie@ua.pt}

\author[M Mañas]{Manuel Mañas$^{3,\clubsuit}$}
\address{$^3$Departamento de Física Teórica, Universidad Complutense de Madrid, Plaza Ciencias 1, 28040-Madrid, Spain \&
 Instituto de Ciencias Matematicas (ICMAT), Campus de Cantoblanco UAM, 28049-Madrid, Spain}
\email{manuel.manas@ucm.es}

\thanks{$^\flat$Acknowledges Centro de Matemática da Universidade de Coimbra (CMUC) -- UID/MAT/00324/2019, funded by the Portuguese Government through FCT/MEC and co-funded by the European Regional Development Fund through the Partnership Agreement PT2020}

\thanks{$^\natural$Acknowledges CIDMA Center for Research and Development in Mathematics and Applications (University of Aveiro) and the Portuguese Foundation for Science and Technology (FCT) within project UIDB/MAT/UID/04106/2020 and UIDP/MAT/04106/2020}

\thanks{$^\clubsuit$Thanks financial support from the Spanish ``Agencia Estatal de Investigación'' research project [PGC2018-096504-B-C33], \emph{Ortogonalidad y Aproximación: Teoría y Aplicaciones en Física Matemática} and [PID2021- 122154NB-I00], \emph{Ortogonalidad y Aproximación con Aplicaciones en Machine Learning y Teoría de la Probabilidad}.}

\keywords{Banded Hessenberg matrices, oscillatory matrices, totally nonnegative matrices, continued fractions, Gauss--Borel factorization, bidiagonal factorization, oscillatory retracted matrices}

\subjclass{42C05,33C45,33C47,15B48,47B39,47B28,47B36}

\begin{abstract}
 Recently a spectral Favard theorem for bounded banded lower Hessenberg matrices that admit a positive bidiagonal factorization was presented. 
 In this paper conditions, in terms of continued fractions, for an oscillatory tetradiagonal Hessenberg matrix to have such positive bidiagonal factorization are found. Oscillatory tetradiagonal Toeplitz matrices are taken as a case study of matrix that admits a positive bidiagonal factorization. Moreover, it is proved that oscillatory banded Hessenberg matrices are organized in rays, with the origin of the ray not having the positive bidiagonal factorization and all the interior points of the ray having such positive bidiagonal factorization.

 \end{abstract}
 
 \maketitle
 
 
 \section{Introduction}
 

In this paper we will study for the tetradiagonal Hessenberg matrix of the form
\begin{align}\label{eq:monic_Hessenberg}
 T&=\left[\begin{NiceMatrix}[columns-width = auto]
 c_0 & 1&0&\Cdots&\\
 b_1&c_1 & 1&\Ddots&&\\
 a_2&b_2&c_2&1&&\\
 0&a_3&b_3&c_3&1&\\
 \Vdots&\Ddots&\Ddots&\Ddots&\Ddots&\Ddots\\&&&&
 \end{NiceMatrix}\right], 
\end{align}
where we assume that $a_n>0$, whether it is possible or not to find a \emph{positive bidiagonal factorization}~(PBF)
 \begin{align}\label{eq:bidiagonal}
 T= L_{1} L_{2} U,
\end{align}
with bidiagonal matrices given by 
\begin{align}\label{eq:LU_J_oscillatory_factors}
L_1=
 \left[\begin{NiceMatrix}
 1 &0&0&\Cdots\\
 \alpha_{2} & 1 &0&\Cdots\\
 0& \alpha_{5} & 1& \Ddots\\
 \Vdots& \Ddots& \Ddots & \Ddots\\
 &&&
 \end{NiceMatrix}\right],
 \ L_2 &=
 \left[\begin{NiceMatrix}
 1 &0&0&\Cdots\\
 \alpha_{3} & 1 &0&\Cdots\\
 0& \alpha_{6} & 1& \Ddots\\
 \Vdots& \Ddots& \Ddots & \Ddots\\&&&
 \end{NiceMatrix}\right] ,
 \ U =
 \left[\begin{NiceMatrix}
 \alpha_1 & 1 &0&\Cdots&\\
 0& \alpha_4 & 1 &\Ddots&\\
 0&0&\alpha_7&\Ddots&\\
 \Vdots& \Ddots& \Ddots &\Ddots &
 \end{NiceMatrix}\right],
\end{align}
with the fulfillment of the following positivity requirement
\begin{align*}
 \alpha_j&>0, & j\in\N.
 \end{align*}
 In \cite{previo} this factorization was shown to be sufficient for a Favard theorem for bounded banded Hessenberg semi-infinite matrices (with $p+2$ diagonals) and the existence of positive measures such that the recursion polynomials are multiple orthogonal polynomials and the Hessenberg matrix is the recursion matrix for this sequence of multiple orthogonal polynomials.
Let us also mention~\cite{Grunbaum_Iglesia} were the authors present an important application of this factorization in order to obtain stochastic bidiagonal factorizations of stochastic Hessenberg matrices.

Finite truncations of these matrices having this PBF are oscillatory matrices. In fact, we will be dealing in this paper with totally nonnegative matrices that are oscillatory matrices and, consequently, we require of some definitions and properties that we are about to present succinctly.

Totally nonnegative (TN) matrices are those with all their minors nonnegative \cite{Fallat-Johnson,Gantmacher-Krein}, and the set of nonsingular TN matrices is denoted by InTN. Oscillatory matrices \cite{Gantmacher-Krein} are totally nonnegative, irreducible \cite{Horn-Johnson} and nonsingular. Notice that the set of oscillatory matrices is denoted by IITN (irreducible invertible totally nonnegative) in \cite{Fallat-Johnson}. An oscillatory matrix~$A$ is equivalently defined as a totally nonnegative matrix $A$ such that for some $n$ we have that $A^n$ is totally positive (all minors are positive). From Cauchy--Binet Theorem one can deduce the invariance of these sets of matrices under the usual matrix product. Thus, following \cite[Theorem 1.1.2]{Fallat-Johnson} the product of matrices in InTN is again InTN (similar statements hold for TN or oscillatory matrices).

We have the important result:
\begin{teo}[Gantmacher--Krein Criterion] \label{teo:Gantmacher--Krein Criterion}
\cite[Chapter 2, Theorem 10]{Gantmacher-Krein} 
A~totally nonnegative matrix $A$ is oscillatory if and only if it is nonsingular and the elements at the first subdiagonal and first superdiagonal are positive.
\end{teo}

Let us discuss the connection of oscillatory matrices with the standard case, i.e. for tridiagonal semi-infinite matrices, named in the literature as Jacobi matrices.

 \begin{defi} 
	Jacobi matrices are tridiagonal real matrices
	\begin{align}\label{eq:Jacobi_matrix_def}
		J&\coloneq \left[\begin{NiceMatrix}[columns-width = auto]
			\mathscr m_0& 1 &0&\Cdots&\\
			\mathscr l_1 &\mathscr m_1& 1& \Ddots&\\
			0&\mathscr l_2&\mathscr m_2&1&\\
			\Vdots& \Ddots& \Ddots & \Ddots&\Ddots\\
			&&&&\\
		\end{NiceMatrix}\right],
	\end{align}
	with $\mathscr l_j>0$, $j = 1 , 2, \dots$ 
	and its finite truncations, that is leading principal submatrices, are
	\begin{align}\label{eq:N_Jacobi}
		J^{[N]}&\coloneq \left[\begin{NiceMatrix}[columns-width = auto]
			\mathscr m_0& 1 &0&\Cdots&&0\\
			\mathscr l_1 &\mathscr m_1& 1& \Ddots&&\Vdots\\
			0&\mathscr l_2&\mathscr m_2&1&&\\
			\Vdots& \Ddots& \Ddots & \Ddots&\Ddots&0\\
			&&&&&1\\
			0 &\Cdots&&0&\mathscr l_N&\mathscr m_N
		\end{NiceMatrix}\right]\in\R^{(N+1)\times (N+1)}, & \Delta_N&\coloneq \det J^{[N]},
	\end{align}
	further important submatrices are
	\begin{align}\label{eq:Nk_Jacobi}
		J^{[N,k]} \coloneq \left[\begin{NiceMatrix}[columns-width = auto]
			\mathscr m_{k}& 1 &0&\Cdots&&0\\
			\mathscr l_{k+1} &\mathscr m_{k+1}& 1& \Ddots&&\Vdots\\
			0&\mathscr l_{k+2}&\mathscr m_{k+2}&1&&\\
			\Vdots& \Ddots& \Ddots & \Ddots&\Ddots&0\\
			&&&&&1\\
			0 &\Cdots&&0&\mathscr l_N&\mathscr m_N
		\end{NiceMatrix}\right]\in\R^{(N+1-k)\times (N+1-k)},
		\
		\Delta_{N,k}\coloneq \det J^{[N,k]} .
	\end{align}
\end{defi}

Regarding oscillatory Jacobi matrices we have

\begin{teo}[Oscillatory Jacobi matrices] \label{teo:Jacobi oscillatory}
 \cite[Chapter XIII,\S 9]{Gantmacher} and \cite[Chapter 2,Theorem 11]{Gantmacher-Krein}.
 A tridiagonal matrix is oscillatory if and only if, 
\begin{enumerate}
 \item The matrix entries of the first subdiagonal and first superdiagonal are positive.
 \item All leading principal minors are positive.
\end{enumerate}
\end{teo}

If the matrix $J$ is bounded, i.e. $\|J\|_\infty <\infty$, all the possible eigenvalues of the submatrices $J^{[N]}$ belong to the disk $D(0,\|J\|_\infty)$.
As all the eigenvalues are real, let us consider those that are negative, and let $b$ be the supreme of the absolute values of all negative eigenvalues. Notice that $b \leqslant \|J\|_\infty$. 
\begin{teo}
	For $s\geqslant b$ the matrix $J_s=J+sI$ is oscillatory.
\end{teo}
\begin{proof}
	Take $s\geqslant b$, then $J_s$ has the eigenvalues of its leading principal submatrices $J^{[N]}_s=J^{[N]}+sI_{N+1}$ all positive. The corresponding characteristic polynomials are $P_{N+1}(x-s)= \det\big( xI_{N+1}-J_s^{[N]}\big)$, so that $\det J_s^{[N]}=(-1)^{N+1}P_{N+1}(-s)$, but as $-s$ is a lower bound for any possible zero of this monic polynomial, we have that $(-1)^{N+1}P_{N+1}(-s)>0$. Hence, the leading principal minors of $J_s$ are all positive and the entries on the subdiagonal a superdiagonal are positive. Thus we conclude, attending to Theorem \ref{teo:Jacobi oscillatory}, that $J_s$ is an oscillatory matrix.
\end{proof}

A very important consequence of this fact, i.e.,
that there exists a positive $s$ such that $J+sI$ is oscillatory 
is that all eigenvalues are simple, and that $P_{N+1}$ interlaces $P_{N}$ and $P^{(1)}_{N+1}$. Indeed, the characteristic polynomial $P_{N+1}(x-s)$ of the oscillatory matrix $J^{[N]}_s$ interlaces the characteristic polynomials of the submatrices $J^{[N]}_s(1)=J^{[N,1]}_s$, i.e. $P^{(1)}_{N+1}(x-s)$, and of $J^{[N]}_s(N+1)=J^{[N-1]}_s$, i.e. $P_N(x-s)$. 
Hence, 
we deduce the positivity of the Christoffel coefficients from the oscillatory character $J^{[N]}_s$. Thus, the spectral Favard theorem, see for example \cite[\S 4.1]{Simon}, is a theorem for bounded Jacobi matrices which are oscillatory up to a appropriate shift. Then, all the interlacing properties follow immediately.

Let us show that for Jacobi matrices the PBF and oscillatory properties are equivalent. 
\begin{pro}\label{pro:Jacobi_PBF}
	A Jacobi matrix is oscillatory if and only if it admits a PBF. 
\end{pro}

\begin{proof}
	Let us assume that the Jacobi matrix $J^{[N,1]}$ is oscillatory. Then, the Gauss--Borel factorization of $J^{[N,1]}$, i.e.
	\begin{align*}
		\left[\begin{NiceMatrix}[columns-width = auto]
			\mathscr m_1& 1 &0&\Cdots&&0\\
			\mathscr l_2 &\mathscr m_2& 1& \Ddots&&\Vdots\\
			0&\mathscr l_3&\mathscr m_3&1&&\\
			\Vdots& \Ddots& \Ddots & \Ddots&\Ddots&0\\
			&&&&&1\\
			0 &\Cdots&&0&\mathscr l_N&\mathscr m_N
		\end{NiceMatrix}\right]= \begin{bNiceMatrix}
			1 &0&\Cdots&&0\\
			\beta_{2} &\Ddots &\Ddots&&\\
			0& \beta_{4} & & &\Vdots\\
			\Vdots& \Ddots& \Ddots & \Ddots&0\\
			0&\Cdots&0&\beta_{2N-2}&1
		\end{bNiceMatrix} \left[\begin{NiceMatrix}
			\beta_1 & 1 &0&\Cdots&0\\
			0& \beta_3 & \Ddots &\Ddots&\Vdots\\
			\Vdots&\Ddots&\beta_5&&0\\
			& & \Ddots &\Ddots &1\\
			0 &\Cdots&&0&\beta_{2N-1}
		\end{NiceMatrix}\right],
	\end{align*}
	leads to $\mathscr m_1=\beta_1$ as well as 
	$\mathscr m_n=\beta_{2n-2}+\beta_{2n-1}$ and $\mathscr l_n=\beta_{2n-2}\beta_{2n-3}$, $\beta_0\coloneq 0$.
	Hence, as $\mathscr l_n>0$, $n\in\{2,3,\dots\}$, we get that $\beta_n>0$ for $n\in\N$. 
	
	If the Jacobi matrix $J^{[N,1]}$ admits a PBF, we deduce that is oscillatory from the Gatmancher--Krein Criterion Theorem \ref{teo:Gantmacher--Krein Criterion}.
	%
\end{proof}


The structure of the paper is as follows. Section~\ref{sec:2} is devoted to introduce truncations and continued fractions needed for the discussion.
Next, in \S~\ref{sec:3} we prove,
in the finite case, the existence of the PBF for tetradiagonal matrices in terms of positive finite continued fractions.
Then, the extension of this PBF for the semi-infinite case is shown to happen when certain nonnegative infinite continued fraction is indeed positive.
In \S~\ref{sec:4} we discuss oscillatory tetradiagonal Toeplitz--Hessenberg matrices as a case study of matrix that admits a PBF.
Finally, in \S~\ref{sec:5} we show that from any given oscillatory tetradiagonal matrix we can construct tetradiagonal matrices with a PBF in several ways. Also, we will find that oscillatory matrices are organized in rays, the origin of the ray is a oscillatory matrix that don't have a PBF and all the interior points of the ray are PBF matrices. We also show that for any PBF tetradiagonal matrix there is a retraction that is oscillatory but without a PBF and vice~versa.

\section{Factorization properties} \label{sec:2}

We now discuss some aspects of the Gauss--Borel factorization and PBF of oscillatory tetradiagonal matrices. For more on Gauss--Borel factorization and orthogonal polynomials see for example~\cite{manas} and references therein.

First, let us introduce some convenient notation for the tetradiagonal case.
Let us denote by $ T^{[N]}=T[\{0,1,\dots,N\}]\in\R^{(N+1)\times(N+1)}$ the $(N+1)$-th leading principal submatrix of the banded Hessenberg matrix $T$: 
\begin{align}\label{eq:Hessenberg_truncation}
	T^{[N]}&\coloneq \begin{bNiceMatrix}[columns-width = auto]
		c_0 & 1&0&\Cdots&&&0\\[-3pt]
		b_1&c_1&1&\Ddots&&&\Vdots\\
		a_2&b_2&c_2&1&&&\\
		0&a_3&b_3&c_3&1&&\\
		\Vdots&\Ddots&\Ddots&\Ddots&\Ddots& \Ddots&0\\
		&&&a_{N-1}&b_{N-1}&c_{N-1}&1\\
		0&\Cdots&&0&a_{N}&b_{N}&c_N
	\end{bNiceMatrix}, & \delta^{[N]}&\coloneq \det T^{[N]}.
\end{align}

Further truncations are
\begin{align*}
	T^{[N,k]} &\coloneq
	\begin{bNiceMatrix}[columns-width = 10pt]
		c_k & 1&0&\Cdots&&&0\\[-3pt]
		b_{k+1}&c_{k+1}&1&\Ddots&&&\Vdots\\
		a_{k+2}&b_{k+2}&c_{k+2}&1&&&\\
		0&a_{k+3}&b_{k+3}&c_{k+3}&1&&\\
		\Vdots&\Ddots&\Ddots&\Ddots&\Ddots& \Ddots&0\\
		&&&a_{N-1}&b_{N-1}&c_{N-1}&1\\
		0&\Cdots&&0&a_{N}&b_{N}&c_N
	\end{bNiceMatrix}
	\in\R^{(N+1-k)\times(N+1-k)}, \ k \in\{0,1,\dots,N\},
\end{align*}
notice that 
$T^{[N,N+1]}\coloneq 1$ and $T^{[N]}=T^{[N,0]}$.


The Gauss--Borel factorization of the leading principal  submatrices $T^{[N]}$ in \eqref{eq:Hessenberg_truncation} is
 \begin{align}\label{eq:Gauss-Borel}
 T^{[N]}&=L^{[N]} U^{[N]}, &
 L^{[N]}&\coloneq\left[\begin{NiceMatrix}[columns-width = auto]
 1 &0&\Cdots&&&0\\
 \mathscr m_1& 1 &\Ddots&&&\Vdots\\
 \mathscr l_2 &\mathscr m_2& 1& &&\\
 0&\mathscr l_3&\mathscr m_3&1&&\\
 \Vdots& \Ddots& \Ddots & \Ddots&\Ddots&0\\
 0 &\Cdots&0&\mathscr l_N&\mathscr m_N&1
 \end{NiceMatrix}\right], & U^{[N]}& \coloneq
 \left[\begin{NiceMatrix}
 \alpha_1 & 1 &0&\Cdots&0\\
 0& \alpha_4 & \Ddots &\Ddots&\Vdots\\
 \Vdots&\Ddots&\alpha_7&&0\\
 & & \Ddots &\Ddots &1\\
 0 &\Cdots&&0&\alpha_{3N+1}
 \end{NiceMatrix}\right].
 \end{align}
 
 \begin{pro}
The Gauss--Borel factorization \eqref{eq:Gauss-Borel} exists if and only if all leading principal minors $\delta^{[n]}$, $n\in\{0,1,\dots,N\}$, of~$T^{[N]}$ are not zero.
For $n\in\N$, the following expressions for the coefficients hold
\begin{align}\label{eq:gauss_borel_a_b_alpha}
 \mathscr l_{n+1}&=\frac{a_{n+1}\delta^{[n-2]}}{\delta^{[n-1]}},&
 \mathscr m_n&=c_n-\frac{\delta^{[n]}}{\delta^{[n-1]}},&
 \alpha_{3n-2}&=\frac{\delta^{[n-1]}}{\delta^{[n-2]}},
 \end{align}
where $\delta^{[-1]}=1$ and $a_1=0$, and we have the following recurrence relation for the determinants
\begin{align}\label{eq:recurrence_delta}
 \delta^{[n]}= a_n\delta^{[n-3]}-b_n\delta^{[n-2]}+c_n\delta^{[n-1]},
\end{align}
is satisfied.

 \end{pro}
 \begin{proof}
 Notice that $\delta^{[N]}=\det T^{[N]}=\det U^{[N]}=\alpha_1\alpha_4\cdots \alpha_{3N+1}$. Hence, we get $\alpha_{3N+1}=\frac{\delta^{[N]}}{\delta^{[N-1]}}$. 
 From the last row of the $LU$ factorization we get
 \begin{align*}
 a_N&=\mathscr l_N\alpha_{3N-5}, &
 b_N&=\mathscr l_N+\mathscr m_N\alpha_{3N-2},&
 c_N&=\mathscr m_N+\alpha_{3N+1}, 
 \end{align*}
so that
 \begin{align*}
\mathscr l_N&=\frac{\delta^{[N-3]}}{\delta^{[N-2]}} a_N, &
\mathscr m_N&=c_N-\frac{\delta^{[N]}}{\delta^{[N-1]}}, 
\end{align*}
and
\begin{align*}
b_N-a_N\frac{\delta^{[N-3]}}{\delta^{[N-2]}} -\frac{\delta^{[N-1]}}{\delta^{[N-2]}}\Big(c_N-\frac{\delta^{[N]}}{\delta^{[N-1]}}\Big)=0,
\end{align*}
that is we have \eqref{eq:recurrence_delta}.
\end{proof}

Note that, in this proof, we could get \eqref{eq:recurrence_delta} by expanding the determinant $\delta^{[N]}$ along the last~row.

\begin{pro}
 \label{pro:abc>0}
 Let us assume that $T^{[N]}$ given in \eqref{eq:Hessenberg_truncation} is an oscillatory matrix.
 Then,
 \begin{align*}
 a_n,b_n,c_n>0.
 \end{align*}
 \end{pro}
 \begin{proof}
 According to the Gantmacher--Krein Criterion, Theorem \ref{teo:Gantmacher--Krein Criterion}, \cite[II.7 Theorem 10]{Gantmacher-Krein}, $b_n>0$ for $n\in\{1,\dots,N\}$. Moreover, as is an invertible totally nonnegative matrix (InTN) according to \cite[page 50, Chapter 2]{Fallat-Johnson} we also need $c_n$ to be positive.
 \end{proof}

 We introduce some auxiliary submatrices that will be instrumental in the following developments.
 
\begin{defi}[Auxiliary submatrices]\label{def:associated_submatrices}
 Given the lower triangular factor $L^{[N]}$, determined by the Gauss--Borel factorization \eqref{eq:Gauss-Borel}, we consider its complementary submatrix, by deleting first row and last column, 
 that we call the auxiliary Jacobi matrix, $J^{[N,1]}=L^{[N]}(\{1\},\{N+1\})\in\R^{N\times N}$ as in \eqref{eq:Nk_Jacobi} with $k=1$.
For $k\in\{0,\dots,N-1\}$, associated with the auxiliary Jacobi matrix $ J^{[N,1]}$ we introduce, as we did with the banded Hessenberg matrix $T^{[N]}$, the principal submatrices $J^{[N,k+1]}$ defined in \eqref{eq:Nk_Jacobi}.
Additionally, we introduce $T^{[N]}_1=T(\{1\},\{N+1\})\in\R^{N\times N}$ as the complementary submatrix obtained by removing the first row and last column of $T^{[N]}$, that is
\begin{align}\label{eq:TN1}
 T^{[N]}_1&\coloneq 
 \begin{bNiceMatrix}[columns-width = .7cm]
 b_{1}&c_{1}&1 &0&\Cdots&&0\\
 a_{2}&b_{2}&c_{2}&1& \Ddots &&\Vdots\\
 0 &a_{3}&b_{3}&c_{3} &1&&\\
 \Vdots&\Ddots&\Ddots&\Ddots&\Ddots& \Ddots&0\\
 &&&&&&1\\
 &&&&a_{N-1}&b_{N-1}&c_{N-1}\\
 0&\Cdots&&&0&a_{N}&b_{N}
 \end{bNiceMatrix}, & 
 \delta^{[N]}_1&\coloneq \det T^{[N]}_1.
\end{align}
Further auxiliary complementary submatrices that we will consider are
\begin{align}\label{eq:TN1k}
 T^{[N,k]}_1&\coloneq \begin{bNiceMatrix}
 b_{k+1}&c_{k+1}&1 &0&\Cdots&&0\\
 a_{k+2}&b_{k+2}&c_{k+2}&1& \Ddots &&\Vdots\\
 0 &a_{k+3}&b_{k+3}&c_{k+3} &1&&\\
 \Vdots&\Ddots&\Ddots&\Ddots&\Ddots& \Ddots&0\\
 &&&&&&1\\
 &&&&a_{N-1}&b_{N-1}&c_{N-1}\\
 0&\Cdots&&&0&a_{N}&b_{N}
 \end{bNiceMatrix},& \delta_1^{[N,k]}&\coloneq \det T^{[N,k]}_1,
\end{align}
so that $T^{[N,0]}_1=T^{[N]}_1$ and $\delta^{[N,0]}_{1}=\delta^{[N]}_{1}$, 
and the upper bidiagonal matrix
\begin{align*}
 U^{[N-1,k]}& =
 \left[\begin{NiceMatrix}[columns-width = auto]
 \alpha_{3k+1} & 1 &0&\Cdots&0\\
 0& \alpha_{3k+4} & \Ddots &\Ddots&\Vdots\\
 \Vdots&\Ddots&\alpha_{3k+7}&&0\\
 & & \Ddots &\Ddots &1\\
 0 &\Cdots&&0&\alpha_{3N-2}
 \end{NiceMatrix}\right].
\end{align*}
\end{defi}

A set of finite continued fractions will be needed in the subsequent analysis. For continued fractions we refer to the interested reader to \cite{Jones,Lorentzen,Wall}.

\begin{defi}[Finite continued fractions]\label{def:finite_contnued_ractions}
 We introduce
 \begin{align*}
 \gaussk[n,k]&\coloneq \mathscr m_{k}-\cfrac{\mathscr l_{k+1}}{\mathscr m_{k+1}-\cfrac{\mathscr l_{k+2}}{\mathscr m_{k+2}-\cfrac{\mathscr l_{k+3}}{\llap{\ensuremath{\ddots}}
 \raisebox{-0.8em}{\ensuremath{\mathscr m_{n-1}-\cfrac{\mathscr l_n}{\mathscr m_n}}}}}}, & n&\in\{k+1,k+2,\dots\}, & \gaussk[k+1,k]&\coloneq\mathscr m_{k} .
 \end{align*}
\end{defi}

We now collect together a number of results regarding factorizations, determinants and relations of the introduced matrices.
The next results are relevant in subsequent developments.

\begin{teo}[Determinants and continued fractions]\label{teo:firstLU}
 Let us assume that $T^{[N]}$ as in 
\eqref{eq:Hessenberg_truncation} is an oscillatory matrix. Then, the following holds:
 \begin{enumerate}
 \item For the triangular factors in \eqref{eq:Gauss-Borel} we have $L^{[N]}, U^{[N]}\in \operatorname{InTN}$. 
 \item The matrix entries of the triangular factors of Gauss--Borel factorization \eqref{eq:Gauss-Borel} of the Hessenberg matrix $T^{[N]}$ are positive: 
$\mathscr l_2,\dots,\mathscr l_N,\mathscr m_1,\dots,\mathscr m_N,\alpha_1,\alpha_4,\dots,\alpha_{3N-2}>0 $.

 \item 
 For $k\in\N$, the recurrence relation
 \begin{align}\label{eq:recurrence_Delta}
 D(n+1)&=\mathscr m_{k+n} D(n)-\mathscr l _{k+n}D(n-1), & n&\in\N,
 \end{align}
 for the initial conditions $D(0)=1,D(1)=\mathscr m_{k}$, 
 has as solution $D(n)=\Delta_{k+n-1,k}$, while for the initial conditions $ D(0)=0$ and $ D(1)=1$ has as solution $\Delta_{k+n-1,k+1}$. The determinants $\Delta_{N,k}$ were defined in \eqref{eq:Nk_Jacobi}.
%
%
 \item The ratio of consecutive determinants are bounded as follows: \begin{align*}
 \dfrac{a_{n+2}}{b_{n+2}}&<\dfrac{\delta^{[n]}}{\delta^{[n-1]}}<c_n,& \dfrac{\mathscr l_{n+1}}{\mathscr m_{n+1}}&<\dfrac{\Delta_{n,1}}{\Delta_{n-1,1}}<\mathscr m_n.
 \end{align*}

\item For $k\in\N$, the continued fraction given in Definition \ref{def:finite_contnued_ractions} is the ratio of the consecutive determinants defined in Equation \eqref{eq:Nk_Jacobi}
\begin{align*}
 \gaussk[n,k]=\frac{\Delta_{n,k}}{\Delta_{n,k+1}}.
\end{align*}
 \item For $k=1$ the determinants in \eqref{eq:Nk_Jacobi} are positive, i.e. $\Delta_{n,1}>0$.
 \end{enumerate}
\end{teo}
\begin{proof}
 \begin{enumerate}
 \item As $T^{[N]}$ is TN, we know it has a Gauss--Borel factorization with totally nonnegative factors, see \cite[Theorem 2.4.1]{Fallat-Johnson}. 
 \item For $n\in\{2,3\cdots,N\}$ we have $a_n=\mathscr l_n\alpha_{3n-5}$ and, consequently, as $a_n\neq 0$ we deduce that $\mathscr l_2,\dots,\mathscr l_N,\alpha_1,\alpha_4,\dots,\alpha_{3N-2}>0$. Moreover,
 \begin{align*}
 \begin{vNiceMatrix}
 \mathscr m_n &1\\
 \mathscr l_{n+1} &\mathscr m_{n+1}
 \end{vNiceMatrix}\geqslant 0,
 \end{align*}
 and as $\mathscr l_{n+1}>0$ we deduce that $\mathscr m_n\mathscr m_{n+1}\neq 0$, and all $\mathscr m_1,\dots,\mathscr m_N>0$.
 \item Expand the determinant $\Delta_{n,k}$ along the last row. One can check that the initial conditions lead to the sequence of determinants.
 \item From \eqref{eq:gauss_borel_a_b_alpha} and $\mathscr m_n>0$ we get $b_n\delta^{[n-2]}>a_n\delta^{[n-3]}$, $c_n\delta^{[n-1]}>\delta^{[n]}$ and the first inequality follows. For the second we use the proof of Proposition \ref{pro:Jacobi_PBF}.
 From the factorization we get $\Delta_{n,1}=\beta_1\cdots\beta_{2n-1}$ and, consequently, 
 $\beta_{2n-1}=\dfrac{\Delta_{n,1}}{\Delta_{n-1,1}}$ and $\beta_{2n-2}=\mathscr m_n-\dfrac{\Delta_{n,1}}{\Delta_{n-1,1}}$. 
 As $\beta_n>0$, $n\in\N$, we deduce that $\mathscr m_n>\dfrac{\Delta_{n,1}}{\Delta_{n-1,1}}$. As for the oscillatory case we require $\Delta_n>0$, the recursion relation \eqref{eq:recurrence_Delta}, i.e., $\Delta_{n,1}=\mathscr m_n\Delta_{n-1,1}-\mathscr l_{n}\Delta_{n-2,1}$, implies that $\mathscr m_n\Delta_{n-1,1}>\mathscr l_{n}\Delta_{n-2,1}$ and the lower bound follows immediately.
 
 \item Use the Euler--Wallis theorem for continued fractions, see for example \cite[Theorem 9.2]{Elaydi}.
 \item The first two determinants are positive, then we apply induction. Let us assume that $\Delta_{n-1,1}>0,$ and that $\Delta_{n,1}=0$. Then, for $k=0$ Equation \eqref{eq:recurrence_Delta} implies that $\Delta_{n+1,1}=-\mathscr l_{n+1}\Delta_{n-1,1}<0$ in contradiction with the fact $\Delta_{n+1,1}\geqslant 0$.
 \end{enumerate}
\end{proof}


\begin{teo}[Factorizations and oscillatory matrices]\label{teo:firstLU2}
 For the submatrices of $T^{[N]}$ and its determinants introduced in Definition \ref{def:associated_submatrices} we find that:
 \begin{enumerate}
 \item The auxiliary Jacobi matrix $ J^{[N,1]}$ is oscillatory.
 \item The following factorizations are fulfilled
 \begin{align}\label{eq:LU_bis}
 T^{[N]}_1&= J^{[N,1]} U^{[N-1]},\\
 \label{eq:trunactions JNK}
 T^{[N,k]}_1-\mathscr m_{k+1} E_{1,1}&=J^{[N,k+1]} U^{[N-1,k]}.
 \end{align}
Moreover, $\delta^{[N]}_1>0$ and we have the following relation between determinants
\begin{align}
 \label{eq:DeltaNdelta}
 \Delta_{N,1}&=\frac{\delta^{[N]}_1}{\delta^{[N-1]}},\\
 \label{eq:trunactions JNK_det}
 \Delta_{N,k+1}&=\alpha_1\cdots \alpha_{3k-2} \frac{\delta^{[N,k]}_{1}-\mathscr m_{k+1}\delta^{[N,k+1]}_{1}}{\delta^{[N-1]}} .
\end{align}
(Recall that $\Delta^{[N,k]}\coloneq\det J^{[N,k]}$, $\delta_1^{[N]}\coloneq\det T^{[N]}_1$ and $\delta_1^{[N,k]}\coloneq\det T^{[N,k]}_1$.)
\item The submatrix $ T_1^{[N]}$ is oscillatory.
\item The submatrices $J^{[N,k+1]}$ and $T_1^{[N,k]}$ are oscillatory. In particular, $\Delta_{N,k+1},\delta^{[N,k]}_1>0$.
\item The following relations are satisfied
\begin{align}
 \notag \Delta_{N,2}\delta^{[N-1]}&=c_0 \delta^{[N,1]}_{1}- a_2\delta^{[N,2]}_{1},\\
 \label{eq:quotientDeltas}
 \frac{\Delta_{N,1}}{\Delta_{N,2}}&=\frac{\delta^{[N]}_1}{c_0 \delta^{[N,1]}_{1}- a_2\delta^{[N,2]}_{1}}.
\end{align}
\item 
The recursion relation in $k$ is satisfied
\begin{align}\label{eq:dual_recursion_Delta}
 \Delta_{N,k+1}=\mathscr m_{k+1}\Delta_{N,k+2}-\mathscr l_{k+2}\Delta_{N,k+3} .
\end{align}
 \end{enumerate}
\end{teo}

\begin{proof}
\begin{enumerate}
 
\item 
According to Theorem \ref{teo:Jacobi oscillatory}, see \cite[Chapter XIII,\S 9]{Gantmacher} and \cite[Chapter 2,Theorem 11]{Gantmacher-Krein}, the Jacobi matrix $ J^{[N,1]}$ is oscillatory if and only if, 
\begin{enumerate}
 \item The matrix entries $\mathscr l_2,\dots ,\mathscr l_N$ are positive.
 \item All leading principal minors $\Delta_{n,1}$ are positive.
\end{enumerate}
As we have seen in previous points, both requirements are satisfied.
\item Equations \eqref{eq:LU_bis} and \eqref{eq:trunactions JNK} follows directly from the Gauss--Borel factorization of $T^{[N]}$. Taking determinants and expanding the determinant along the first row we get \eqref{eq:trunactions JNK_det}.
 From Equation \eqref{eq:LU_bis} we conclude that
$\alpha_1\alpha_4\cdots \alpha_{3n-2} \Delta_n=\det T^{[n]}_1$.
As previously said all $\alpha_1,\alpha_4,\dots, \alpha_{3n-2}>0$ and $\Delta_{n,1}>0$. Therefore, $\delta^{[N]}_1\neq 0$.
\item Given that the matrix $T^{[N]}_1$ belongs to InTN with $a_n,b_n,c_n>0$, see Proposition \ref{pro:abc>0}, the Gantmacher--Krein Criterion, Theorem \ref{teo:Gantmacher--Krein Criterion}, leads to the oscillatory character of this submatrix. 
\item If a matrix $A$ is oscillatory then so is any submatrix $A[\boldsymbol \alpha]$ for any contiguous subset of indexes $\boldsymbol{\alpha}$, 
see \cite[Chapter 2, \S 7]{Gantmacher-Krein} and \cite[Corollary 2.6.7]{Fallat-Johnson}. Then, $J^{[N,k+1]}=J^{[N]}(\{k+1,\dots,N\})$ and $T_1^{[N,k]}=T_1^{[N]}(\{k+1,\dots,N\})$ are oscillatory and, consequently, $\Delta_{N,k+1}=\det J^{[N]}(\{k+1,\dots,N\})>0$ and $\delta_1^{[N,k]}=\det T_1^{[N]}(\{k+1,\dots,N\})>0$.
\item Put $k=1$ in \eqref{eq:trunactions JNK_det} and recall that $ \alpha_1 =c_0$ and $\alpha_1\mathscr l_2=a_2$. For Equation \eqref{eq:quotientDeltas} use \eqref{eq:DeltaNdelta}.
\item Expand the determinants along the first row.

\end{enumerate}
\end{proof}

A set of convergent infinite continued fractions are important in what follows.

\begin{defi}[Infinite continued fraction and tails]\label{def:infinite_continued_fraction}
 We introduce the following infinite continued fraction
 \begin{align}\label{eq:infinite_continued_fraction}
 \gaussk[1]&\coloneq \mathscr m_{1}-\cfrac{\mathscr l_{2}}{\mathscr m_{2}-\cfrac{\mathscr l_{3}}{\mathscr m_{3}-\raisebox{-.33\height}{\(\ddots\)}}}
 \end{align}
and its tails
\begin{align*}
 \gaussk[k+1]&\coloneq \mathscr m_{k+1}-\cfrac{\mathscr l_{k+2}}{\mathscr m_{k+2}-\cfrac{\mathscr l_{k+3}}{\mathscr m_{k+3}-\raisebox{-.33\height}{\(\ddots\)}}}, & k&\in\N.
\end{align*}
\end{defi}
\begin{coro}
The infinite continued fraction in \eqref{eq:infinite_continued_fraction} can be computed as the following large $N$ limit ratio
 \begin{align}\label{eq:continued_fraction}
 \gaussk[1]&=\lim_{N\to\infty}\frac{\delta^{[N]}_1}{c_0 \delta^{[N,1]}_{1}- a_2\delta^{[N,2]}_{1}}
 \end{align}
of determinants given in \eqref{eq:TN1} and \eqref{eq:TN1k}.
\end{coro}
\begin{proof}
 Direct consequence of \eqref{eq:quotientDeltas}.
\end{proof}

Now, an important result follows regarding the behavior of these infinite continued fractions.

\begin{teo}[Infinite continued fractions]
For the infinite continued fractions $\gaussk[1]$ given in~\eqref{eq:infinite_continued_fraction} we~have:
\begin{enumerate}
 \item For $k\in\N_0$, the sequences $\left\{\gaussk[n,k]\right\}_{n=k+1}^\infty$ of the finite continued fractions given in Definition \ref{def:finite_contnued_ractions} are positive and strictly decreasing.
 \item 
The infinite continued fraction $\gaussk[1]$ converges and is nonnegative. 
\item The tails converge and are positive, i.e. $\gaussk[k+1]>0$ for $k\in\N$.
\end{enumerate}
\end{teo}
\begin{proof}
\begin{enumerate}
 \item The positivity follows at once from the positivity of $\Delta_{N,k}$. 
From \eqref{eq:dual_recursion_Delta} we have
\begin{align}\label{eq:Deltas}
 \frac{ \Delta_{N+1,k+1}}{ \Delta_{N+1,k+2}}&=\mathscr m_{k+1}
 -\cfrac{\mathscr l_{k+2}}{\frac{ \Delta_{N+1,k+2}}{ \Delta_{N+1,k+3}}},&
 \frac{ \Delta_{N,k+1}}{ \Delta_{N,k+2}}&=\mathscr m_{k+1}
 -\cfrac{\mathscr l_{k+2}}{\frac{ \Delta_{N,k+2}}{ \Delta_{N,k+3}}}.
\end{align}
As $m_k,\Delta_{N,k+1}>0$ the inequality
\begin{align}\label{eq:first_inequality}
 \frac{\Delta_{N+1,k+1}}{\Delta_{N+1,k+2}}< \frac{\Delta_{N,k+1}}{\Delta_{N,k+2}},
\end{align}
can be written 
\begin{align*}
 \mathscr m_{k+1}
 -\cfrac{\mathscr l_{k+2}}{\frac{ \Delta_{N+1,k+2}}{ \Delta_{N+1,k+3}}}<\mathscr m_{k+1}
 -\cfrac{\mathscr l_{k+2}}{\frac{ \Delta_{N,k+2}}{ \Delta_{N,k+2}}},
\end{align*}
where we have used \eqref{eq:Deltas}. Therefore, \eqref{eq:first_inequality} is equivalent to the inequality
\begin{align*}
 \frac{\Delta_{N+1,k+2}}{\Delta_{N+1,k+3}}< \frac{\Delta_{N,k+2}}{\Delta_{N,k+3}}.
\end{align*}
Hence, if for $k=N-2$ the inequality 
\begin{align}\label{eq:inequality_basic}
 \frac{\Delta_{N+1,N-1}}{\Delta_{N+1,N}}< \frac{\Delta_{N,N-1}}{\Delta_{N,N}}
\end{align}
is fulfilled, the inequality \eqref{eq:first_inequality} will hold. But, 
\begin{align*}
 \frac{\Delta_{N+1,N-1}}{\Delta_{N+1,N}}&=\mathscr m_N-\frac{\mathscr l_{N+1}}{\mathscr m_{N+1}},& \frac{\Delta_{N,N-1}}{\Delta_{N,N}}&=\mathscr m_N,
\end{align*}
and \eqref{eq:inequality_basic} satisfied.


\item Obvious from the previous result, any positive decreasing sequence is convergent to a nonnegative number.
\item For $k\in\N$ we have $\gaussk[k]=\mathscr m_k-\frac{\mathscr l_{k+1}}{\gaussk[k+1]}$, $\mathscr l_{k+1}>0$, so that the convergence of $\gaussk[k]$ requires $\gaussk[k+1]>0$. 
\end{enumerate}
\end{proof}

%
%


\section{Positive bidiagonal factorization of tetradiagonal Hessenberg matrices} \label{sec:3}
Now we discuss how the Gauss--Borel factorization can be used to find a bidiagonal factorization of the banded Hessenberg matrix. This will lead to the appearance of continued fractions in our theory.
\begin{lemma}
The factorization of any lower triangular matrix of the form 
\begin{align*}
 L^{[N]}&=\left[\begin{NiceMatrix}[columns-width = auto]
 1 &0&\Cdots&&&0\\
 \mathscr m_1& 1 &\Ddots&&&\Vdots\\
 \mathscr l_2 &\mathscr m_2& 1& &&\\
 0&\mathscr l_3&\mathscr m_3&1&&\\
 \Vdots& \Ddots& \Ddots & \Ddots&\Ddots&0\\
 0 &\Cdots&0&\mathscr l_N&\mathscr m_N&1
 \end{NiceMatrix}\right], 
 \end{align*}
 into bidiagonal factors, i.e.,
\begin{align}\label{eq:L1L2U_truncated}
 L^{[N]}&= L^{[N]}_{1} L_{2}^{[N]} ,&
 L^{[N]}_{1}&=
 \begin{bNiceMatrix}
 1 &0&\Cdots&&0\\
 \alpha_{2} &\Ddots &\Ddots&&\\
 0& \alpha_{5} & & &\Vdots\\
 \Vdots& \Ddots& \Ddots & \Ddots&0\\
 0&\Cdots&0&\alpha_{3N-1}&1
 \end{bNiceMatrix},& L_{2}^{[N]}&=
 \begin{bNiceMatrix}
 1 &0&\Cdots&&0\\
 \alpha_{3} &\Ddots &\Ddots&&\\
 0& \alpha_{6} & & &\Vdots\\
 \Vdots& \Ddots& \Ddots & \Ddots&0\\
 0&\Cdots&0&\alpha_{3N}&1
 \end{bNiceMatrix},
\end{align}
is uniquely determined in terms of $\alpha_2$, with
\begin{align}\label{eq:continued _fraction_bidiagonals}
 \alpha_{3n}&= \mathscr m_{n}-\cfrac{\mathscr l_{n}}{\mathscr m_{n-1}-\cfrac{\mathscr l_{n-1}}{\mathscr m_{n-2}-\cfrac{\mathscr l_{n-2}}{\llap{\ensuremath{\ddots}}
 \raisebox{-0.8em}{\ensuremath{\mathscr m_{2}-\cfrac{\mathscr l_2}{\mathscr m_1-\alpha_2}}}}}}, & \alpha_{3n-1}&= \cfrac{\mathscr l_{n}}{\mathscr m_{n-1}-\cfrac{\mathscr l_{n-1}}{\mathscr m_{n-2}-\cfrac{\mathscr l_{n-2}}{\llap{\ensuremath{\ddots}}
 \raisebox{-0.8em}{\ensuremath{\mathscr m_{2}-\cfrac{\mathscr l_2}{\mathscr m_1-\alpha_2}}}}}}.
\end{align}
The factorization exists if and only if $\alpha_{3n}\neq 0$ for $n\in\{1,\dots,N-1\}$.
\end{lemma}
\begin{proof}
 The factorization \eqref{eq:L1L2U_truncated} implies that
 \begin{align*}
 \mathscr m_n&=\alpha_{3n-1}+\alpha_{3n}, & n&\in\{1,\dots, N\}, &\mathscr l_n&=\alpha_{3n-1}\alpha_{3n-3},&n&\in\{2,\dots, N\}.
 \end{align*}
These can be solved recursively as
\begin{align*}
 \alpha_3&=\mathscr m_1-\alpha_2, & \alpha_5&=\frac{\mathscr l_2}{\mathscr m_1-\alpha_2}, \\ \alpha_6&=\mathscr m_2-\frac{\mathscr l_2}{\mathscr m_1-\alpha_2}, &
 \alpha_8&=\cfrac{\mathscr l_3}{\mathscr m_2-\frac{\mathscr l_2}{\mathscr m_1-\alpha_2}}, \\
 \alpha_9&=\mathscr m_3-\cfrac{\mathscr l_3}{\mathscr m_2-\frac{\mathscr l_2}{\mathscr m_1-\alpha_2}},& \alpha_{10}&=\cfrac{\mathscr l_4}{\mathscr m_3-\cfrac{\mathscr l_3}{\mathscr m_2-\frac{\mathscr l_2}{\mathscr m_1-\alpha_2}}},
\end{align*}
and the result follows by induction.
Hence, for a given $\alpha_2$ the factorization exists if and only if $\alpha_{3n}\neq 0$, for $n\in\{1,\dots,N-1\}$.
\end{proof}

\begin{pro}\label{pro:bidiagonalL}
For each $\alpha_2< \gaussk[N,1]$, with $\gaussk[N,1]$ the finite continued fraction in Definition \ref{def:finite_contnued_ractions},
the factorization of $L^{[N]}$ into bidiagonal factors 
\begin{align*}
 L^{[N]}&= L^{[N]}_{1} L_{2}^{[N]} ,&
 L^{[N]}_{1}&=
 \begin{bNiceMatrix}
 1 &0&\Cdots&&0\\
 \alpha_{2} &\Ddots &\Ddots&&\\
 0& \alpha_{5} & & &\Vdots\\
 \Vdots& \Ddots& \Ddots & \Ddots&0\\
 0&\Cdots&0&\alpha_{3N-1}&1
 \end{bNiceMatrix},& L_{2}^{[N]}&=
 \begin{bNiceMatrix}
 1 &0&\Cdots&&0\\
 \alpha_{3} &\Ddots &\Ddots&&\\
 0& \alpha_{6} & & &\Vdots\\
 \Vdots& \Ddots& \Ddots & \Ddots&0\\
 0&\Cdots&0&\alpha_{3N}&1
 \end{bNiceMatrix},
\end{align*}
with $ \alpha_3,\alpha_5,\alpha_6,\alpha_8,\dots,\alpha_{3n-1},\alpha_{3n}>0$, exists and is unique.
 If $\alpha_2\in[0,\gaussk[N,1])$ then $L^{[N]}_{1} , L_{2}^{[N]}\in\operatorname{InTN}$.
\end{pro}

\begin{proof}
In the solution provided by Equation \eqref{eq:continued _fraction_bidiagonals} we require that $ \alpha_3,\alpha_5,\alpha_6,\alpha_8,\dots,\alpha_{3N-1},\alpha_{3N}>0$. Let us proceed step by step, firstly if $\alpha_2<\mathscr m_1$ we see that
$\alpha_3,\alpha_5>0$. In the next step, we get that if $\alpha_2<\mathscr m_1$ and $\alpha_2<\mathscr m_1-\frac{\mathscr l_2}{\mathscr m_2}$ we have $\alpha_3,\alpha_5,\alpha_6,\alpha_8>0$. 
Notice that as the sequence $\gaussk[N,1]>0$ is decreasing $\mathscr m_1-\frac{\mathscr l_2}{\mathscr m_2}<\mathscr m_1$ and only one condition is needed.
Then, in the next step we conclude that what is needed for $\alpha_3,\alpha_5,\alpha_6,\alpha_8,\alpha_9,\alpha_{10}>0$ is $\alpha_2<\mathscr m_1-\cfrac{\mathscr l_2}{\mathscr m_2-\frac{\mathscr l_3}{\mathscr m_3}}$. 
Finally, induction implies the result.
\end{proof}

\begin{teo}[Positive bidiagonal factorization in the finite case]
 Let us assume that the matrix $ T^{[N]}$ given in~\eqref{eq:Hessenberg_truncation}
 is oscillatory. Then, each 
 $\alpha_2< \gaussk[N,1]$ determines a positive sequence 
$\{\alpha_1,\alpha_3,\alpha_4,\alpha_5,\dots,$ $\alpha_{3N+1}\}$
 such that the factorization
 \begin{align*} 
 T^{[N]}= L^{[N]}_{1} L_{2}^{[N]} U^{[N]},
 \end{align*}
 with bidiagonal matrices given by
 \begin{align*}
\hspace{-.105cm}
L^{[N]}_{1} =
 \begin{bNiceMatrix}
 1 &0&\Cdots&&0\\
 \alpha_{2} &\Ddots &\Ddots&&\\
 0& \alpha_{5} & & &\Vdots\\
 \Vdots& \Ddots& \Ddots & \Ddots&0\\
 0&\Cdots&0&\alpha_{3N-1}&1
 \end{bNiceMatrix},
 L_{2}^{[N]} =
 \begin{bNiceMatrix}
 1 &0&\Cdots&&0\\
 \alpha_{3} &\Ddots &\Ddots&&\\
 0& \alpha_{6} & & &\Vdots\\
 \Vdots& \Ddots& \Ddots & \Ddots&0\\
 0&\Cdots&0&\alpha_{3N}&1
 \end{bNiceMatrix},
 U^{[N]} =
 \left[\begin{NiceMatrix}
 \alpha_1 & 1 &0&\Cdots&0\\
 0& \alpha_4 & \Ddots &\Ddots&\Vdots\\
 \Vdots&\Ddots&\alpha_7&&0\\
 & & \Ddots &\Ddots &1\\
 0 &\Cdots&&0&\alpha_{3N+1}
 \end{NiceMatrix}\right] ,
 \end{align*}
is satisfied.
When $\alpha_2\in[0,\gaussk[N,1])$ each bidiagonal factor is InTN.
\end{teo}
\begin{proof}
Consequence of ii) in Theorem \ref{teo:firstLU} and Proposition \ref{pro:bidiagonalL}.
\end{proof}


\begin{teo}[PBF in the semi-infinite case]
 Let us assume that the banded Hessenberg matrix $ T$ in~\eqref{eq:monic_Hessenberg} is oscillatory.
 Then, for each $\alpha_2<\gaussk[1]$, with $\gaussk[1]$ the infinite continued fraction in \eqref{eq:infinite_continued_fraction}, there exist a unique positive sequence $\{\alpha_1,\alpha_3,\alpha_4,\dots\}$ such that the PBF \eqref{eq:bidiagonal}, \eqref{eq:LU_J_oscillatory_factors} holds.
 If $\alpha_2\in [0,\gaussk[1])$ then $L_1,L_2,U\in \operatorname{InTN}$.
Moreover, we have
the following relations for the matrix entries, 
 \begin{align*}
 \left\{\begin{aligned}
 c_n & = \alpha_{3n+1} + \alpha_{3n} + \alpha_{3n-1} , \\
 b_n & = \alpha_{3n} \alpha_{3n-2}+ \alpha_{3n-1} \alpha_{3n-2}+ \alpha_{3n-1}\alpha_{3n-3}, \\
 a_n & = \alpha_{3n-1}\alpha_{3n-3} \alpha_{3n-5}.
 \end{aligned}\right.
 \end{align*}
\end{teo}

It is known that the infinite continued fraction $\gaussk[1]$ in \eqref{eq:infinite_continued_fraction} could be zero, and this is an important issue in the constructions of the spectral measure representation for the banded Hessenberg matrix $T$,
as $\alpha_2$ can not be taken as a positive number (cf. \cite{previo}).

Notice that in \cite{Aptekarev_Kaliaguine_VanIseghem} this was taken for granted, and that in the hypergeometric case \cite{nuestro2,Lima-Loureiro} and the Jacobi--Piñeiro in the semi-band \cite{Aptekarev_Kaliaguine_VanIseghem,nuestro1} is also true that $\alpha_2>0$. 

\section{Oscillatory Toeplitz tetradiagonal matrices}\label{sec:4}
We discuss now the uniform case that appears when 
\begin{align}\label{eq:Toeplitz}
 a_n&=a>0, & b_n&=b\geqslant 0, &c_n&=c\geqslant 0.
\end{align}
That is, the Hessenberg matrix $T$ is a banded Toeplitz matrix
\begin{align}\label{eq:Topelitz}
 T&=\left[\begin{NiceMatrix}[columns-width = auto]
 c& 1&0&\Cdots&\\[-3pt]
 b&c & 1&\Ddots&&\\
 a&b&c&1&&\\
 0&a&b&c&1&\\
 \Vdots&\Ddots&\Ddots&\Ddots&\Ddots&\Ddots\\&&&&
\end{NiceMatrix}\right].
\end{align}
\begin{pro}[Edrei--Schoenberg]\label{pro:toeplitz}
 The Toeplitz matrix \eqref{eq:Topelitz} is oscillatory if and only if there exists $\betaup_1\geqslant\betaup_2\geqslant\betaup_3>0$ such that 
 \begin{align}\label{eq:a_betas}
 a&=\betaup_1\betaup_2\betaup_3,&
 b&=\betaup_1\betaup_2+\betaup_1\betaup_3+\betaup_2\betaup_3,&
 c&=\betaup_1+\betaup_2+\betaup_3.
 \end{align}
\end{pro}
\begin{proof}
 According to the Edrei--Schoenberg Theorem, see \cite{Edrei,Schoenberg}, the matrix $T$ is TN if and only if the generating function
\begin{align*}
 f(t)=1+ct+bt^2+at^3
\end{align*}
can be written as
\begin{align*}
 f(t)&=(1+\betaup_1 t)(1+\betaup_2 t)(1+\betaup_3 t), & \betaup_1\geqslant\betaup_2\geqslant\betaup_3\geqslant 0.
\end{align*}
In terms of these $\betaup$'s we find \eqref{eq:a_betas}.
As $a>0$ we must have $\betaup_1\geqslant\betaup_2\geqslant\betaup_3>0$. Hence $b>0$ and the Gantmacher--Krein criterion, see Theorem \ref{teo:Gantmacher--Krein Criterion}, leads to the oscillatory character of the Toeplitz matrix.
\end{proof}

Now we will show that all tetradiagonal Toeplitz matrices \eqref{eq:Topelitz} are oscillatory if and only if admits a PBF.

\begin{pro}\label{pro:toeplitz_2}
 If $T$ is a oscillatory banded Toeplitz matrix as in
\eqref{eq:Topelitz} with $\betaup_1>\betaup_2>\betaup_3> 0$, then the determinants $\delta^{[N]}=\det T^{[N]}$ are explicitly given in terms of $\{\betaup_1,\betaup_2,\betaup_3\}$ as follows:
 \begin{align}\label{eq:toelitz_determinant}
 \delta ^{[n]}=\frac{\betaup_1^{n+2}}{(\betaup_1-\betaup_2)(\betaup_1-\betaup_3)}
 +\frac{\betaup_2^{n+2}}{(\betaup_2-\betaup_1)(\betaup_2-\betaup_3)}+\frac{\betaup_3^{n+2}}{(\betaup_3-\betaup_1)(\betaup_3-\betaup_2)}.
 \end{align}
\end{pro}

\begin{proof}
	The determinants $\delta^{[n]}=\det T^{[n]}$ are subject to the recursion relation
	\begin{align}\label{eq:constant_recursion}
		\delta^{[n]}-c \delta^{[n-1]}+b\delta^{[n-2]}-a\delta^{[n-3]}=0,
	\end{align}
	being the initial conditions: $\delta^{[-2]}=\delta^{[-1]}=0$ and $\delta^{[0]}=1$.
 Following the theory of recursion relations, see for example \cite{Elaydi}, we consider the so called characteristic polynomial 
\begin{align*}
p(\lambda) \coloneq\lambda^3-c\lambda^2+b\lambda-a,
\end{align*}
and notice that $p(\lambda)=\lambda^3f\big(-\frac{1}{\lambda}\big)$; hence, the characteristic roots are $\betaup_1,\betaup_2,\betaup_3>0$. If the roots are distinct, i.e., simple, then the general solution to the recursion \eqref{eq:constant_recursion} will be
\begin{align*}
C_1 \betaup_1^n+C_2\betaup_2^n+C_3\betaup_3^n
\end{align*}
for some constants $C_1,C_2,C_3$ 
determined by the initial conditions:
\begin{align*}
 \begin{bNiceMatrix}
 1& 1& 1\\
 \frac{1}{\betaup_1} & \frac{1}{\betaup_2}& \frac{1}{\betaup_3}\\
 \frac{1}{\betaup_1^2} & \frac{1}{\betaup_2^2}& \frac{1}{\betaup_3^2}
 \end{bNiceMatrix}\begin{bNiceMatrix}
 C_1 \\ C_2 \\ C_3
\end{bNiceMatrix}=\begin{bNiceMatrix}
1 \\0 \\0
\end{bNiceMatrix},
\end{align*}
so that \eqref{eq:toelitz_determinant} holds.
\end{proof}


\begin{coro}
For $T$ as in Proposition \ref{pro:toeplitz}, for the large $N$ ratio asymptotics of the determinants we find
\begin{align}\label{eq:limit_ratio}
 \lim_{N\to\infty}\frac{\delta^{[N]}}{\delta^{[N-1]}}=\betaup_1.
\end{align}
\end{coro}

\begin{proof}
 In the case of distinct characteristic roots $\betaup_1>\betaup_2>\betaup_3> 0$ is a direct consequence of Proposition \ref{pro:toeplitz_2}. 
 
 When the two smaller characteristic roots coincide $\betaup_2=\betaup_3$ the general solution will be
\begin{align*}
 C_1 \betaup_1^n+(C_2+C_3n)\betaup_2^n
\end{align*}
and the large $N$ ratio asymptotics of the determinant do not change. When the largest characteristic root is degenerate, with multiplicity two or three, then asymptotically the determinant will have a dominant term $q(n) \betaup_1^n$ with a polynomial $q$ such that $\deg q=1,2$ in the determinants and \eqref{eq:limit_ratio} is recovered.
\end{proof}

\begin{teo}[Infinite continued fractions and harmonic mean]
The continued fraction considered in~\eqref{eq:infinite_continued_fraction} for an oscillatory tetradiagonal Toeplitz $T$ as \eqref{eq:Topelitz} is a half of the harmonic mean of the two largest characteristic roots, i.e
 \begin{align*}
 \gaussk[1]&=\frac{\betaup_1\betaup_2}{\betaup_1+\betaup_2}.
 \end{align*}
\end{teo}
\begin{proof}
In order to compute the continued fraction $\gaussk[1]$ according to \eqref{eq:continued_fraction} we only require the use the determinants $\delta^{[N]}_1$ of 
\begin{align*}
 T^{[N]}_1&= 
 \begin{bNiceMatrix}
 b&c&1 &0&\Cdots&&0\\
 a&b&c&1& \Ddots &&\Vdots\\
 0 &a&b&c &1&&\\
 \Vdots&\Ddots&\Ddots&\Ddots&\Ddots& \Ddots&0\\
 &&&&&&1\\
 &&&&a&b&c\\
 0&\Cdots&&&0&a&b
 \end{bNiceMatrix}
\in\R^{N\times N},
\end{align*}
as in this Toeplitz case $\delta_1^{[N,k]}=\delta_1^{[N-k]}$. These determinants are subject to the following uniform recursion relation
\begin{align}\label{eq:constant_recursion_1}
 \delta_1^{[n]}-b \delta_1^{[n-1]}+ac\delta_1^{[n-2]}-a^2\delta_1^{[n-3]}=0,
\end{align}
as follows from an expansion along the last column. 
The initial conditions are 
$\delta_1^{[-2]}=\delta_1^{[-1]}=0$ and $\delta_1^{[0]}=1$.
The characteristic roots $\gammaup_1,\gammaup_2,\gammaup_3$ are the zeros of
\begin{align*}
q(t)=t^3-bt^2+ac t-a^2,
\end{align*}
and we find that $p(t)=-\frac{t^3}{a^2}q\big(\frac{a}{t}\big)$ or $q(t)=-\frac{t^3}{a}p\big(\frac{a}{t}\big)$. Hence, the characteristic roots are $\gammaup=\frac{a}{\betaup}$, that arranged in decreasing order can be written as follows
\begin{align*}
 \gammaup_1&=\betaup_1\betaup_2, &\gammaup_2&=\betaup_1\betaup_3,& \gammaup_3&=\betaup_3\betaup_2.
\end{align*}
Let us assume that the roots are distinct, i.e., simple, the other degenerate cases can be treated similarly. Then, the general solution to the recursion \eqref{eq:constant_recursion_1} will be
\begin{align*}
 C_1 \gammaup_1^n+C_2\gammaup_2^n+C_3\gammaup_3^n,
\end{align*}
for some constants $C_1,C_2,C_3$ 
determined by the initial conditions:
\begin{align*}
 \begin{bNiceMatrix}
 1& 1& 1\\
 \frac{1}{\gammaup_1} & \frac{1}{\gammaup_2}& \frac{1}{\gammaup_3}\\
 \frac{1}{\gammaup_1^2} & \frac{1}{\gammaup_2^2}& \frac{1}{\gammaup_3^2}
 \end{bNiceMatrix}\begin{bNiceMatrix}
 C_1 \\ C_2 \\ C_3
 \end{bNiceMatrix}=\begin{bNiceMatrix}
 1 \\0 \\0
 \end{bNiceMatrix}
\end{align*}
and, proceeding as above, we find
\begin{align*}
\delta_1 ^{[n]}
 &=\frac{\betaup_1^{n+1}\betaup_2^{n+1}}{(\betaup_2-\betaup_3)(\betaup_1-\betaup_3)}+
 \frac{\betaup_1^{n+1}\betaup_3^{n+1}}{(\betaup_3-\betaup_2)(\betaup_1-\betaup_2)}+
 \frac{\betaup_2^{n+1}\betaup_3^{n+1}}{(\betaup_3-\betaup_1)(\betaup_2-\betaup_1)}.
\end{align*}
Hence, according to \eqref{eq:continued_fraction} with $c_0 = c$ and $a_2 = a$ we have
\begin{align*}
 \gaussk[1]
 &
 =\lim_{n\to\infty}\frac{\dfrac{\betaup_1^{n+1}\betaup_2^{n+1}}{(\betaup_2-\betaup_3)(\betaup_1-\betaup_3)}}{
 (\betaup_1+\betaup_2+\betaup_3)\dfrac{\betaup_1^{n}\betaup_2^{n}}{(\betaup_2-\betaup_3)(\betaup_1-\betaup_3)}-
 \betaup_1\betaup_2\betaup_3 \dfrac{\betaup_1^{n-1}\betaup_2^{n-1}}{(\betaup_2-\betaup_3)(\betaup_1-\betaup_3)}
}\\&=\frac{\betaup_1^2\betaup_2^2}{(\betaup_1+\betaup_2+\betaup_3)\betaup_1\betaup_2-\betaup_1\betaup_2\betaup_3}
 ,
\end{align*}
and after some manipulations the representation for $\gaussk[1]$ follows.
\end{proof}
\begin{teo}
	A tetradiagonal Toeplitz matrix is oscillatory if and only if admits a PBF.
\end{teo}


%
%
%

\section{Retractions and tails} \label{sec:5}
We will show that from any given oscillatory tetradiagonal matrix we can construct tetradiagonal matrices with PBF in several ways. Also, we will find that oscillatory matrices are organized in rays, the origin of the ray is a oscillatory matrix that do not have a PBF and all the interior points of the ray have a PBF.

 Given a TN matrix $A$, then $A+sE_{1,1}$ is also TN for $s\geqslant -\frac{\det A}{\det A(\{1\})}$, see \cite[Section 9.5]{Fallat-Johnson} on retractions of TN matrices and in particular the proof of \cite[Lemma 9.5.2]{Fallat-Johnson}.
When $s$ is a negative number this is known as a retraction. As we will show, for any matrix with a PBF there is a retraction that is oscillatory but with $\alpha_2=0$ and vice versa. Following \cite{Fallat-Johnson}, we use the notation $E_{1,2}(s)\coloneq I+s E_{1,2}$ for the bidiagonal matrix with a nonzero contribution only possibly at the entry in the second row and first column.
\begin{teo}[Retractions and PBF]
\begin{enumerate}
 \item Given an oscillatory tetradiagonal matrix $T$ as in \eqref{eq:monic_Hessenberg} then the tetradiagonal matrix
 
 \begin{align}\label{eq:Ts}
 T_s&=E_{1,2}(s)T=\left[\begin{NiceMatrix}[columns-width = .7cm]
 c_0 & 1&0&\Cdots&\\[-3pt]
 b_1+sc_0&c_1+s & 1&\Ddots&&\\
 a_2&b_2&c_2&1&&\\
 0&a_3&b_3&c_3&1&\\
 \Vdots&\Ddots&\Ddots&\Ddots&\Ddots&\Ddots\\&&&&
 \end{NiceMatrix}\right], 
 \end{align}
has a PBF for $s>-\gaussk[1]$.
 \item Given a tetradiagonal matrix $T$ as in \eqref{eq:monic_Hessenberg} that admits a PBF, i.e., with $\gaussk[1]>0$, then
 the tetradiagonal matrix 
 \begin{align*} 
 \tilde T&=E_{1,2}\Big(-\gaussk[1]\Big)T=\left[\begin{NiceMatrix}[columns-width = 1cm]
 c_0 & 1&0&\Cdots&\\[-3pt]
 b_1-\gaussk[1]c_0&c_1-\gaussk[1]& 1&\Ddots&&\\
 a_2&b_2&c_2&1&&\\
 0&a_3&b_3&c_3&1&\\
 \Vdots&\Ddots&\Ddots&\Ddots&\Ddots&\Ddots\\&&&&
 \end{NiceMatrix}\right], 
 \end{align*}
is a oscillatory matrix that does not admit a PBF, i.e. $\tilde {\gaussk}[1]=0$.
\end{enumerate}
\end{teo}

\begin{proof}
\begin{enumerate}
 \item The Jacobi matrix $J_s^{[N,1]}=J^{[N,1]}+s E_{1,1}$ is TN for $s\geqslant -\frac{\Delta_{N,1}}{\Delta_{N,2}}$ and InTN for $s> -\frac{\Delta_{N,1}}{\Delta_{N,2}}=-\gaussk[N,1]$. Thus, attending to Theorem \ref{teo:Gantmacher--Krein Criterion}, is an oscillatory matrix for $s> -\gaussk[N,1]$. Then, the corresponding lower unitriangular matrix $L^{[N]}_s$ that has $J_s^{[N,1]}$ as complementary submatrix, $L^{[N]}_s(\{1\},\{N+1\})=J_s^{[N,1]}$ (obtained by deleting first row and last column), is InTN for $s>-\gaussk[N,1]$. This is a consequence of \cite[Lemma 3.3.4]{Fallat-Johnson}. 
The continued fraction $\gaussk_s[N,1]$, corresponding to the oscillatory Jacobi matrix $J^{[N,1]}_s$, is $ \gaussk[N,1]+s$.

Now, let us consider the banded Hessenberg matrix
by defining $T^{[N]}_s=L^{[N]}_s U^{[N]}$, which is clearly InTN for $s> -\gaussk[N,1]$ as its factors are. A direct computation shows that
\begin{align*}
 T_s^{[N]}= \begin{bNiceMatrix}[columns-width = auto]
 c_0 & 1&0&\Cdots&&&0\\[-3pt]
 b_1+sc_0&c_1+s&1&\Ddots&&&\Vdots\\
 a_2&b_2&c_2&1&&&\\
 0&a_3&b_3&c_3&1&&\\
 \Vdots&\Ddots&\Ddots&\Ddots&\Ddots& \Ddots&0\\
 &&&a_{N-1}&b_{N-1}&c_{N-1}&1\\
 0&\Cdots&&0&a_{N}&b_{N}&c_N
 \end{bNiceMatrix}.
\end{align*}
Observe that $\mathscr m_1=\frac{b_1}{c_0}=\gaussk[2,1]$, and recall that $\{\gaussk[n,1]\}_{n=2}^\infty$ is positive and decreasing sequence and, consequently, $b_1+sc_0>0$ for $s>-\gaussk[N,1]$ and $N\in\{2,3,\dots\} $. Therefore, using the Gantmacher--Krein Criterion, 
we conclude that $T_s^{[N]}$ is oscillatory.
Finally, for the large $N$ limit with $s>-\gaussk[1]$, the matrix 
 $T_s$ is oscillatory and has $\gaussk_s[1]>0$.
 \item The retraction $\tilde{ J}^{[N,1]}=J^{[N,1]}-\gaussk[1]E_{1,1}$ of the Jacobi matrix of $T^{[N]}$is oscillatory. This is a direct consequence of the fact that $0<\gaussk[1]<\gaussk[n,1]$, for all $n\in\N$. 
 The associated finite continued fraction is $\tilde{\gaussk}[N,1]=\gaussk[N,1]-\gaussk[1]$.
 To continue, let us consider as in the previous discussion the unitriangular matrix $\tilde L^{[N]}$ such that $\tilde J^{[N,1]}=\tilde L^{[N]}(\{1\},\{N+1\})$ is a complementary submatrix, deleting first row and last column. This triangular matrix is InTN, and the matrix $\tilde T^{[N]}=\tilde L^{[N]} U^{[N]}$ is oscillatory, as well, with associated continued fraction $\tilde{\gaussk}[N,1]=\gaussk[N,1]-\gaussk[1]$. Thus, in the large $N$ limit the semi-infinite banded Hessenberg matrix $\tilde T$ is oscillatory with $\tilde {\gaussk}[1]=0$, i.e. no PBF exits.
\end{enumerate}
\end{proof}
\begin{coro}
 If the tetradiagonal matrix $T$ given in \eqref{eq:monic_Hessenberg} has $\gaussk[1]=0$, then $T_s$ as in \eqref{eq:Ts}has a PBF for~$s>0$.
\end{coro}

The next result is based on the fact that the tails $\gaussk[2],\gaussk[3],\dots$ of the continued fraction $\gaussk[1]$ are positive.

\begin{teo}[Tails and positive bidiagonal factorization]\label{theorem:tails and retractions}
 If $T$ is an oscillatory banded Hessenberg matrix as in \eqref{eq:monic_Hessenberg} then the matrices
 \begin{align*}
 T^{(2)}&\coloneq
 \left[\begin{NiceMatrix}
 c_{1}-\frac{b_1}{c_0} & 1&0&\Cdots&\\[-3pt]
 b_{2}-\frac{a_2}{c_0}&c_{2}& 1&\Ddots&&\\
 a_{3}&b_{3}&c_{3}&1&&\\
 0&a_{4}&b_{4}&c_{4}&1&\\
 \Vdots&\Ddots&\Ddots&\Ddots&\Ddots&\Ddots\\&&&&
 \end{NiceMatrix}\right], \\
 T^{(k+1)}&\coloneq
 \left[\begin{NiceMatrix}[columns-width = 1cm]
 \alpha_{3k+1}& 1&0&\Cdots&&\\[-3pt]
 (c_{k+1}+\alpha_{3k+4})\alpha_{3k+1}&c_{k+1}& 1&\Ddots&&\\
 a_{k+2}&b_{k+2}&c_{k+2}&1&&\\
 0&a_{k+3}&b_{k+3}&c_{k+3}&1&\\
 \Vdots&\Ddots&\Ddots&\Ddots&\Ddots&\Ddots\\
 &&&&&
 \end{NiceMatrix}\right], &k&\in\{2,3,\dots\} ,
 \end{align*}
have a PBF with associated continued fractions $\gaussk[2]$ and $\gaussk[k+1]$, $k\in\{2,3,\dots\}$, respectively.
\end{teo}

\begin{proof}
For $k\in\N$, the tail $\gaussk[k+1]$ is the continued fraction of the Jacobi matrix 
\begin{align}
 J^{(k+1)}&\coloneq \left[\begin{NiceMatrix}[columns-width = 1cm]
 \mathscr m_{k+1}& 1 &0&\Cdots&\\
 \mathscr l_{k+2} &\mathscr m_{k+2}& 1& \Ddots&\\
 0&\mathscr l_{k+3}&\mathscr m_{k+3}&1&\\
 \Vdots& \Ddots& \Ddots & \Ddots&\Ddots\\
 &&&&
 \end{NiceMatrix}\right],
\end{align}
that is oscillatory and a submatrix of
\begin{align*}
 L^{(k+1)}&=
 \left[\begin{NiceMatrix}[columns-width = 1cm]
 1 &0&\Cdots&&\\
 \mathscr m_{k+1}& 1 &\Ddots&&\\
 \mathscr l_{k+2} &\mathscr m_{k+2}& 1& &\\
 0&\mathscr l_{k+3}&\mathscr m_{k+3}&1&\\
 \Vdots& \Ddots& \Ddots & \Ddots&\Ddots\\
 &&&&
 \end{NiceMatrix}\right] ,
\end{align*}
with all its leading principal submatrices InTN.
We introduce the upper triangular matrix
\begin{align*}
 U^{(k+1)}=\left[\begin{NiceMatrix}[columns-width = auto]
 \alpha_{3k+1} & 1 &0&\Cdots&\\
 0& \alpha_{3k+4} & 1 &\Ddots&\\
 0&0&\alpha_{3k+7}&\Ddots&\\
 \Vdots& \Ddots& \Ddots &\Ddots &
 \end{NiceMatrix}\right], 
\end{align*}
with all its leading principal submatrices InTN,
and the corresponding tetradiagonal matrix $T^{(k+1)}=L^{(k+1)}U^{(k+1)}$ is
\begin{align*}
 T^{(k+1)}= 
\left[\begin{NiceMatrix}[columns-width = .7cm]
 c_{k}-\mathscr m_{k} & 1&0&\Cdots&\\[-3pt]
 b_{k+1}-\mathscr l_{k+1}&c_{k+1}& 1&\Ddots&&\\
 a_{k+2}&b_{k+2}&c_{k+2}&1&&\\
 0&a_{k+3}&b_{k+3}&c_{k+3}&1&\\
 \Vdots&\Ddots&\Ddots&\Ddots&\Ddots&\Ddots\\&&&&
 \end{NiceMatrix}\right], 
\end{align*}
with $c_{k}-\mathscr m_{k} =\alpha_{3k+1}>0$ and $b_{k+1}-\mathscr l_{k+1}=\mathscr m_{k+1}\alpha_{3k+1}=(c_{k+1}+\alpha_{3k+4})\alpha_{3k+1}>0$ is, according to the Gantmacher--Krein Criterion, an oscillatory matrix, having its continued fraction $\gaussk^{(k+1)}[1]$ equal to the tail $\gaussk[k+1]$ that is positive. To get $T^{(2)}$ recall that
$\alpha_1=c_0$, $\mathscr m_1\alpha_1=b_1$ and $\mathscr l_2\alpha_1=a_2$. 
\end{proof}

We end the paper with a result that ensures, given any oscillatory tetradiagonal matrix, the finding of associated tetradiagonal matrices with a PBF.
\begin{teo}\label{theorem:spectral_representation2}
	Let us assume that
the Hessenberg matrix $T$ in~\eqref{eq:monic_Hessenberg} is oscillatory. Then, the tetradiagonal matrices
		\begin{align*}
			\check T &\coloneq \left[\begin{NiceMatrix}[columns-width = auto]
				b_1 & 1&0&\Cdots&\\[-3pt]
				a_2c_1&b_2&1&\Ddots&&\\
				a_2a_3&a_3c_2&b_3&1&&\\
				0&a_3a_4&a_4c_3&b_4&1&\\
				\Vdots&\Ddots&\Ddots&\Ddots&\Ddots& \Ddots\\
				&&&&
			\end{NiceMatrix}\right], &\\	\check T^{[k]} &\coloneq 
			\left[\begin{NiceMatrix}
				b_{k+1} -\mathscr m_{k+1}& 1&0&\Cdots&&\\
				a_{k+2}c_{k+1}&b_{k+2}&1&\Ddots&&\\
				a_{k+2}a_{k+3}&a_{k+3}c_{k+2}&b_{k+3}&1&&\\
				0&a_{k+3}a_{k+4}&a_{k+4}c_{k+3}&b_{k+4}&1&\\
				\Vdots&\Ddots&\Ddots&\Ddots&\Ddots& \Ddots
				\\
			\end{NiceMatrix}\right], & k&\in\N,
		\end{align*}
	admit a PBF.
\end{teo}
\begin{proof}
	Theorem \ref{teo:firstLU2} ensures that the submatrix	$	T^{[N]}_{1}$ and its 
	submatrices $T^{[N,k]}_{1}$ are oscillatory. Moreover, Equation \eqref{eq:trunactions JNK} and the fact that 
	$ J^{[N,k]}$ is oscillatory and $U^{[N-1,k]}$ is InTN (notice that the product is InTN and the elements in the first superdiagonal and subdiagonal are positive, then use Gatmacher--Krein Criterion, Theorem \ref{teo:Gantmacher--Krein Criterion}) imply that the retraction $T^{[N,k]}_{1}-\mathscr l_{k+1}	E_{1,1}$ is oscillatory as well. These matrices are upper Hessenberg, so that transposition will transform then in lower Hessenberg
	\begin{align*}
		\big(T^{[N]}_1\big)^\top&\coloneq \begin{bNiceMatrix}[columns-width = 1cm]
			b_1 & a_2&0&\Cdots&&&0\\[-3pt]
			c_1&b_2&a_3&\Ddots&&&\Vdots\\
			1&c_2&b_3&a_4&&&\\
			0&1&c_3&b_4&a_5&&\\
			\Vdots&\Ddots&\Ddots&\Ddots&\Ddots& \Ddots&0\\
			&&&&&b_{N-1}&a_{N}\\
			0&\Cdots&&0&1&a_{N-1}&b_N
		\end{bNiceMatrix}, \\
		\big(T^{[N,k]}_1-m_{k+1}E_{1,1}\big)^\top&\coloneq \begin{bNiceMatrix}[columns-width = 1.5cm]
			b_{k+1} -m_{k+1}& a_{k+2}&0&\Cdots&&&0\\[-3pt]
			c_{k+1}&b_{k+2}&a_{k+3}&\Ddots&&&\Vdots\\
			1&c_{k+2}&b_{k+3}&a_{k+4}&&&\\
			0&1&c_{k+3}&b_{k+4}&a_{k+5}&&\\
			\Vdots&\Ddots&\Ddots&\Ddots&\Ddots& \Ddots&0\\
			&&&&&b_{N-1}&a_{N}\\
			0&\Cdots&&0&1&a_{N-1}&b_N
		\end{bNiceMatrix}.
	\end{align*}
	They are not normalized to be monic on the first superdiagonal but on the second subdiagonal. However, a similarity $T\mapsto ATA^{-1}$ transformation by the diagonal matrix 
\begin{align*}
		A=\diag(1,a_2,a_2a_3,\dots, a_2a_3\cdots a_N)
\end{align*}
	leads to monic banded Hessenberg matrix, 
	\begin{align*}
		\check T^{[N-1]}\coloneq A\big(T^{[N]}_1\big)^\top A^{-1}=\begin{bNiceMatrix}[columns-width = auto]
			b_1 & 1&0&\Cdots&&&0\\[-3pt]
			a_2c_1&b_2&1&\Ddots&&&\Vdots\\
			a_2a_3&a_3c_2&b_3&1&&&\\
			0&a_3a_4&a_4c_3&b_4&1&&\\
			\Vdots&\Ddots&\Ddots&\Ddots&\Ddots& \Ddots&0\\
			&&&&&b_{N-1}&1\\
			0&\Cdots&&0&a_{N-1}a_N&a_Nc_{N-1}&b_N
		\end{bNiceMatrix}
	\end{align*}
	which happens to be oscillatory. Analogously, a similarity $T\mapsto ATA^{-1}$ transformation by the diagonal matrix 
	$	A_k=\diag(1,a_{k+2},a_{k+2}a_{k+3},\dots, a_{k+2}a_{k+3}\cdots a_N)$
	leads to monic banded Hessenberg matrix, 
	\begin{align*}
		\check T^{[N-1,k]}&\coloneq A_k\big(T^{[N,k]}_1-\mathscr m_{k+1}E_{1,1}\big)^\top A_k^{-1}\\&=\begin{bNiceMatrix}[columns-width = 1.5cm]
			b_{k+1 }-\mathscr m_{k+1}& 1&0&\Cdots&&&0\\
			a_{k+2}c_{k+1}&b_{k+2}&1&\Ddots&&&\Vdots\\
			a_{k+2}a_{k+3}&a_{k+3}c_{k+2}&b_{k+3}&1&&&\\
			0&a_{k+3}a_{k+4}&a_{k+4}c_{k+3}&b_{k+4}&1&&\\
			\Vdots&\Ddots&\Ddots&\Ddots&\Ddots& \Ddots&0\\
			&&&&&&1\\
			0&\Cdots&&0&a_{N-1}a_N&a_Nc_{N-1}&b_N
		\end{bNiceMatrix} ,
	\end{align*}
retraction that happens to be oscillatory.
Moreover, according to \eqref{eq:LU_bis} and \eqref{eq:trunactions JNK} these matrices admit the factorizations
	\begin{align}\label{eq:fac_tilde}
	\begin{aligned}
			\check T^{[N-1]}&=\check L^{[N-1]}\check J^{[N]}, & \check L^{[N-1]}&\coloneq A\big(U^{[N-1]}\big)^\top A^{-1}, &
			\check J^{[N,1]}&=A\big(J^{[N,1]}\big)^\top A^{-1},\\
			\check T^{[N-1,k]}&=\check L^{[N-1,k]}\check J^{[N,1]}, & \check L^{[N-1,k]}&\coloneq A_k\big(U^{[N-1,k]}\big)^\top A_k^{-1}, &
			\check J^{[N,k+1]}&=A_k\big(J^{[N,k+1]}\big)^\top A_k^{-1},
		\end{aligned}
	\end{align}
with (recalling that $	\alpha_1 =c_0$)
	\begin{gather*}
		\begin{aligned}
			\check L^{[N-1]}	&=
			\begin{bNiceMatrix}
				c_0&0&\Cdots&&0\\
				a_{2} &\alpha_4 &\Ddots&&\\
				0& a_3& \alpha_7& &\Vdots\\
				\Vdots& \Ddots& \Ddots & \Ddots&0\\
				0&\Cdots&0&a_N&\alpha_{3N-2}
			\end{bNiceMatrix}, &\check J^{[N,1]}&=\left[\begin{NiceMatrix}[columns-width = auto]
				\mathscr m_1& \frac{\mathscr l_2}{a_2} &0&\Cdots&&0\\
				a_2 &\mathscr m_2& \frac{\mathscr l_3}{a_3} & \Ddots&&\Vdots\\
				0& a_3&\mathscr m_3&\frac{\mathscr l_4}{a_4}&&\\
				\Vdots& \Ddots& \Ddots & \Ddots&\Ddots&0\\
				&&&&& \frac{\mathscr l_N}{a_N} \\
				0	&\Cdots&&0&a_N&\mathscr m_N
			\end{NiceMatrix}\right],
		\end{aligned}\\
		\hspace*{-1.25cm}\begin{aligned}
			\check L^{[N-1,k]}	&=
			\begin{bNiceMatrix}[columns-width = 1cm]
				\alpha_{3k+1}&0&\Cdots&&0\\
				a_{k+2} &\alpha_{3k+4} &\Ddots&&\\
				0& a_{k+3}& \alpha_{3k+7}& &\Vdots\\
				\Vdots& \Ddots& \Ddots & \Ddots&0\\
				0&\Cdots&0&a_N&\alpha_{3N-2}
			\end{bNiceMatrix}, &\check J^{[N,k+1]}&=\left[\begin{NiceMatrix}[columns-width = 1cm]
				\mathscr m_{k+1}& \frac{\mathscr l_{k+2}}{a_{k+2}} &0&\Cdots&&0\\
				a_{k+2} &\mathscr m_{k+2}& \frac{\mathscr l_{k+3}}{a_{k+3}} & \Ddots&&\Vdots\\
				0& a_{k+3}&\mathscr m_{k+3}& \frac{\mathscr l_{k+4}}{a_{k+4}}&&\\
				\Vdots& \Ddots& \Ddots & \Ddots&\Ddots&0\\
				&&&&& \frac{\mathscr l_N}{a_N} \\
				0	&\Cdots&&0&a_N&\mathscr m_N
			\end{NiceMatrix}\right].
		\end{aligned}
	\end{gather*}
Now, the Jacobi matrices involved in these factorizations are oscillatory and consequently admit a PBF. Hence, the existence of a PBF follows.
\end{proof}

\section*{Conclusions and outlook}

In \cite{previo} it was shown that for bounded banded Hessenberg matrices (having $p+2$ diagonals) with a PBF we can ensure a spectral Favard theorem \cite{Simon} for multiple orthogonal polynomials. This is the first extension of the spectral Favard theory to multiple orthogonality.\footnote{
For an account of multiple orthogonal polynomials see \cite{nikishin_sorokin} and \cite{Ismail} (chapter written by Van Assche), see \cite{afm} for its description in terms of a Gauss--Borel factorization and its connection with integrable systems and \cite{bfm} for a discussion of Pearson equations and Christoffel formulas for general Christoffel/Geronimus transformations.} Therefore, is important to discuss how the oscillatory character of a banded Hessenberg matrix is related to this PBF. We have shown that for the well known tridiagonal case, the corresponding Jacobi matrix when conveniently shifted is oscillatory, and that oscillatory and PBF coincide in this case. Next step is the tetradiagonal case. In this work we study the existence of a PBF for a generic tetradiagonal oscillatory Hessenberg matrix. 

In terms of continued fractions, in the finite case we prove the existence of the bidiagonal positive factorization in the tetradiagonal scenario and for the infinite case we present a bound for the existence of the bidiagonal positive factorization.
We take the oscillatory Toeplitz matrices as a case study and prove that they admit PBF. Also, it is shown that whenever a oscillatory tetradiagonal matrix is not PBF there are several ways of finding associated oscillatory matrices that have a PBF.

In \cite{proximo} for the tetradiagonal case, and corresponding multiple orthogonal polynomials in the step-line with two weights, the PBF factorization is given in terms of the values of the orthogonal polynomials of type I and II at $0$ and, consequently, an spectral interpretation of the Darboux transformation is given.

For the future, we will like to understand what happens with Hessenberg matrices with more diagonals. Preliminary attempts show that the role of continued fractions must be replaced by more general objects, maybe branched continued fractions. Similar questions open for the constant Toeplitz matrix, do a oscillatory pentadiagonal Toeplitz matrix always admit a PBF? What happens when the banded recursion matrix has several superdiagonals as well as subdiagonals?

\end{document}